\documentclass[12pt,a4paper]{article}
\usepackage[utf8]{inputenc}
\usepackage{amsmath, amssymb, amsthm}
\usepackage{graphicx}

\usepackage{float} 
\usepackage{hyperref}  
\usepackage{cleveref}
\usepackage{geometry}
\usepackage{subcaption}
\newtheorem{theorem}{Theorem}
\newtheorem{lemma}{Lemma}

\theoremstyle{remark}
\newtheorem*{remark}{Remark}
\theoremstyle{definition}
\newtheorem{definition}{Definition}
\DeclareMathOperator{\Slash}{Slash}

\usepackage[T1]{fontenc}

\title{Topology of a Uniform Spanning Tree on a Cylinder}

\author{%
  Nikita Kalinin\(^{1,2}\) \and
  Denis Rakhmankin\(^{1,2,3}\)
}

\date{%
  \(^{1}\) Guangdong Technion -- Israel Institute of Technology, 241 Daxue Road, Shantou, Guangdong, China, 515063, \\
  \(^{2}\) Technion -- Israel Institute of Technology, Haifa, 3200003, Israel\\
  \(^{3}\) Saint Petersburg State University, Saint Petersburg, Russia\\[1ex]
  \texttt{nikita.kalinin@gtiit.edu.cn},\;
  \texttt{s.9166556309@ya.ru}\\[1ex]
  \today
}

\begin{document}

\maketitle

\begin{abstract}
We study uniform spanning trees (USTs) on the discrete cylinder
\(G_{n,m}=C_n\times P_m\), in the regime where the circumference
\(n\) is fixed and the length \(m\) tends to infinity.
Using Wilson's algorithm, with the initial root on one boundary ring
and the first walk started from the other, we single out a
\emph{trunk} \(L\subset T\): the first loop-erased random-walk path
connecting the two boundary rings. We prove that the tree-distance from
any fixed vertex to this trunk has an exponential tail with constants
independent of \(m\). As a consequence, the longest branch attached to the trunk is
at most logarithmic in the length of the cylinder with high probability. 

Our motivation comes from the Abelian sandpile model on cylinders, and in
particular from the step-like, or ``ladder'', avalanche-size distributions
observed numerically by Eckmann--Nagnibeda--Perriard in
\cite{EckmannNagnibedaPerriard2023}. Via Dhar's burning algorithm, recurrent
sandpile configurations correspond to spanning trees, suggesting that the
geometry of a typical UST may influence how avalanches propagate along the
cylinder. 
The branch estimates above, their wired analogue, and the exponential estimate
for the interface separating vertices whose paths to the sink pass through
opposite boundary rings are intended as a first step towards isolating
geometric UST observables that may be relevant to these plateau phenomena.
\end{abstract}

\section{Introduction}

The Abelian sandpile model (ASM) on cylindrical graphs provides a playground for studying self-organized criticality via its avalanche statistics. In their recent work on sandpiles on discrete cylinders,
Eckmann--Nagnibeda--Perriard
\cite{EckmannNagnibedaPerriard2023} considered the ASM on a cylinder
of fixed circumference and large length, with dissipation through the
two boundary rings.  They uncovered a striking ``ladder'' structure in the distribution of avalanche sizes: for a wide range of intermediate sizes, avalanches appear with nearly constant probability, forming plateaus rather than a simple power-law decay. This phenomenon seems specific to cylindrical geometries and is not observed in the same form on planar boxes or in one dimension.

Dhar's burning algorithm relates recurrent configurations of the ASM on a finite graph to spanning trees of the underlying graph with a sink added~ \cite{Dhar1999AbelianSandpile}. More precisely, the burning bijection identifies recurrent sandpile
configurations with spanning trees of the graph rooted at the sink.
This makes the geometry of the associated spanning tree a natural
object to examine when studying avalanche propagation.
On infinite graphs, this correspondence underlies the construction of infinite-volume
sandpile measures and avalanche-size distributions via wired uniform spanning forests;
see, for example, \cite{AthreyaJarai2004}. The numerical observations of
Eckmann--Nagnibeda--Perriard raise a natural conceptual question: \emph{which large-scale geometric
features of the associated spanning trees could control, or at least reflect,  avalanche
propagation on a long cylinder?}

In this paper we take a first step in this direction by studying the geometry of a uniformly random spanning tree on a finite cylindrical graph. We consider the graph
\[
G_{n,m} = C_n \times P_m,
\]
the Cartesian product of a cycle of length \(n\) and a path with \(m\)
vertices, together with its wired version, obtained by attaching a
sink to the two boundary rings. Let $T$ be a uniformly chosen spanning tree of \(G_{n,m}\). A simple structural observation, made precise below, is that one can choose a trunk, namely a simple path in $T$ connecting one boundary ring to the other,
such that the tree-distance from any fixed vertex to this trunk has an
exponential tail with constants independent of the cylinder length. At the level of
Wilson's algorithm, the key estimate is that each increment added after the
trunk has a uniformly exponential length tail. 

We formalize this picture by choosing a trunk $L\subset T$ as the first loop-erased path in Wilson's algorithm connecting the two boundary rings of the cylinder;  in particular, $L$ intersects every ring of the cylinder (see Figure~\ref{fig:tree}). A \emph{branch} is a connected component of \(T\setminus L\), together with its unique attaching vertex on \(L\). Its length is the maximal tree-distance from this attaching vertex to a vertex of the component. 

Our first result is an exponential union bound for the event that some branch
has length at least \(l\): there exist constants $C = C(n) > 0$ and $\theta = \theta(n) \in (0,1)$, depending only on the circumference $n$ and independent of the length $m$, such that for all $m \ge 2$ and all $l \ge 0$,
\[
\mathbb{P}\bigl(\text{UST on } G_{n,m} \text{ has a branch off } L \text{ of length} \ge l\bigr)
\;\le\; C\,m\,  (n-1)\, \theta^{l}.
\]
This follows from the fixed-vertex exponential tail by taking a union bound
over the vertices outside \(L\), and implies that the maximal branch length is
\(O_n(\log m)\), that is, at most \(C(n)\log m\), with high probability. Thus, for fixed circumference, a typical tree may be viewed as a
one-dimensional backbone spanning the cylinder, decorated by branches
whose individual length tails are uniformly exponential and whose
maximal length is at most logarithmic in \(m\) with high probability.

Our second result is the corresponding wired statement, where the two boundary
rings are attached to a sink. This is the version naturally related to
recurrent sandpiles through Dhar's burning bijection. It says that, with high
probability, there exists a sink-trunk crossing the cylinder such that all
branches off this sink-trunk have logarithmic length. More precisely, the
failure probability is bounded by \(C(n)m^3\theta^\ell\). Thus, taking
\[
\ell=
\left\lceil
\frac{3\log m+A}{|\log\theta(n)|}
\right\rceil
\]
gives a sink-trunk whose maximal branch length is less than \(\ell\) with
probability at least \(1-C(n)e^{-A}\). In particular, for every
\(\varepsilon>0\), taking \(A=\varepsilon\log m\) gives
\[
\ell=
\left\lceil
\frac{(3+\varepsilon)\log m}{|\log\theta(n)|}
\right\rceil
\]
and failure probability at most \(C(n)m^{-\varepsilon}\). Hence the maximal
branch length is \(O_n(\log m)\) with high probability.

Besides the local geometry of branches off a trunk, we also introduce a more global
cut-type observable of a spanning tree on the wired cylinder, which we call the \emph{LR--Slash}; see Figure~\ref{fig:slash}.
For a spanning tree \(T\) of the wired cylinder \(G_{n,m}^{\mathrm{s}}\), we decompose the vertices of \(G_{n,m}\) into \emph{left} and \emph{right} classes according to their last step before reaching the sink: a vertex belongs to the left class if its unique path to the sink \(s\) uses, as its final edge, an edge from \(R_0\) to \(s\); it belongs to the right class if the final edge comes from \(R_{m-1}\). The LR--Slash of \(T\) is the set of real cylinder edges, namely edges of \(G_{n,m}\), whose endpoints lie in different left/right classes.

Our third result, Theorem~\ref{thm:Slash}, shows that the size of the
LR--Slash also has an exponential tail, with constants depending only
on \(n\): there exist \(C=C(n)>0\) and
\(\delta=\delta(n)\in(0,1)\) such that, for all \(m\ge2\) and all
\(l>0\),
\[
\mathbb{P}\!\left(\, \bigl|\Slash_{\mathrm{LR}}(T)\bigr|>l \,\right)\;\le\; C(n)\,\delta^{\,l}.
\]
Thus, in addition to the logarithmic high-probability control of the maximal
branch length, the number of real cylinder edges whose endpoints have opposite left/right labels is also tightly controlled.

From the sandpile perspective, these observables suggest a possible geometric mechanism behind avalanche propagation on long cylinders. Under the burning bijection, large-scale features of the spanning tree may constrain how avalanches move along the cylinder. In particular, the LR--Slash is a natural candidate for a tree-theoretic counterpart of the ``blocker'' structures appearing in the empirical discussion of \cite{EckmannNagnibedaPerriard2023}. We do not prove a quantitative relation between branch lengths, LR--Slash size, and avalanche-size distributions here; making such a connection precise remains an open problem.

Our proofs rely on Wilson's algorithm, which generates uniform spanning
trees from loop-erased random walks \cite{Wilson1996}, together with
elementary random-walk estimates on finite-width cylinders.
The finite-width cylindrical geometry is essential: every vertex lies
within distance \(O(n)\) of any path connecting the two boundary rings,
because such a path intersects every ring. This uniform transverse bound controls the
loop-erased increments that attach to the trunk. The remainder of the paper introduces the notation, states the main estimates, and proves the trunk, wired-trunk, and LR--Slash bounds.

\paragraph{Related work.} 
Uniform spanning trees on planar lattices and on graphs embedded in surfaces have been studied extensively.  
In particular, large-scale connection probabilities and loop-erased random walk intensities for USTs on graphs on surfaces (including annuli) are analyzed in detail in
\cite{KenyonWilson2015}. 
From an enumerative viewpoint, spanning trees on graphs with cyclic symmetry (in particular, cylindrical ``cobweb'' lattices) are counted in 
\cite{YanZhang2011}
and, in greater generality, in 
\cite{ZhangLuJin2024}.
Finite-size corrections and boundary-condition dependence for the spanning-tree partition function on square lattices, including cylindrical boundary conditions, are obtained in
 \cite{IzmailianKenna2015}.

Random walks on discrete cylinders with large bases and their connection to
random interlacements are studied in 
\cite{Windisch2010}.
There the focus is on the microscopic structure of the vacant set left by a long random walk on
$G_N\times\mathbb{Z}$ and its convergence to a random interlacement model.
Conceptually, this provides another probabilistic perspective on long
cylindrical geometries, complementary to our trunk/branch and slash
description for uniform spanning trees.

From a dynamical point of view, the Abelian sandpile model can be realized as a factor of an algebraic \(\mathbb{Z}^d\)-action on a compact abelian group, the harmonic model, as shown in \cite{SchmidtVerbitskiy2009}. Their results give a global measure-theoretic description of critical and dissipative sandpile dynamics.

The works cited above approach spanning trees, random walks, and sandpile dynamics on cylindrical or related geometries from several complementary directions: enumeration, partition functions, loop-erased random-walk observables, random-walk vacant sets, and measure-theoretic sandpile dynamics. Our emphasis is different. We study two concrete geometric features of a UST on a finite-width cylinder:
the branch structure off a Wilson trunk, and the left--right interface in the
wired model.

The branch estimates record a simple but useful consequence of fixed circumference: after the trunk is fixed, all later Wilson increments have uniformly exponential length tails. The LR--Slash estimate is more global. It gives an exponential tail for the number of real cylinder edges whose
endpoints have opposite left/right labels. To the best of our knowledge, the LR--Slash observable in this form,
defined by the final boundary entrance of the tree path to a single wired
sink, has not previously been studied on finite cylinders. Thus the contribution of the paper is twofold: first, to make explicit the trunk-and-branches geometry of USTs on finite-width cylinders, and second, to prove an exponential tail bound for the size of this left--right interface. 

\section{Notation and main statements}
\begin{definition}
For \(n\ge3\) and \(m\ge2\), the \emph{cylindrical graph}
\(G_{n,m}\), shown in Figure~\ref{fig:a}, is the Cartesian product
\[
G_{n,m}=C_n\times P_m .
\]
Equivalently, its vertex set is
\[
V(G_{n,m})=
\{(i,j): i\in\mathbb{Z}_n,\ j\in\{0,\dots,m-1\}\},
\]
and its edge set consists of
\[
\begin{aligned}
&\text{ring edges:}
&&\bigl\{\{(i,j),(i+1,j)\}: i\in\mathbb{Z}_n,\ 
j=0,\dots,m-1\bigr\},\\
&\text{path edges:}
&&\bigl\{\{(i,j),(i,j+1)\}: i\in\mathbb{Z}_n,\ 
j=0,\dots,m-2\bigr\}.
\end{aligned}
\]
We denote the rings of the cylinder by
\[
R_k=\{(i,k): i\in\mathbb{Z}_n\},
\qquad k=0,1,\dots,m-1.
\]
The two boundary rings are \(R_0\) and \(R_{m-1}\).

The \emph{wired cylindrical graph} \(G_{n,m}^{\mathrm{s}}\) is
obtained from \(G_{n,m}\) by adding one new vertex \(s\), called the
sink, and connecting it to all vertices of the two boundary rings:
\[
V(G_{n,m}^{\mathrm{s}}):=V(G_{n,m})\cup\{s\},
\]
and
\[
E(G_{n,m}^{\mathrm{s}})
:=
E(G_{n,m})
\cup
\bigl\{\{s,v\}: v\in R_0\cup R_{m-1}\bigr\}.
\]
\end{definition}

We will mainly be interested in the finite-width regime, where \(n\)
is fixed and \(m\) is large.

\begin{figure}[H]
  \centering
  \begin{subfigure}{0.48\textwidth}
    \centering
    \includegraphics[width=\linewidth]{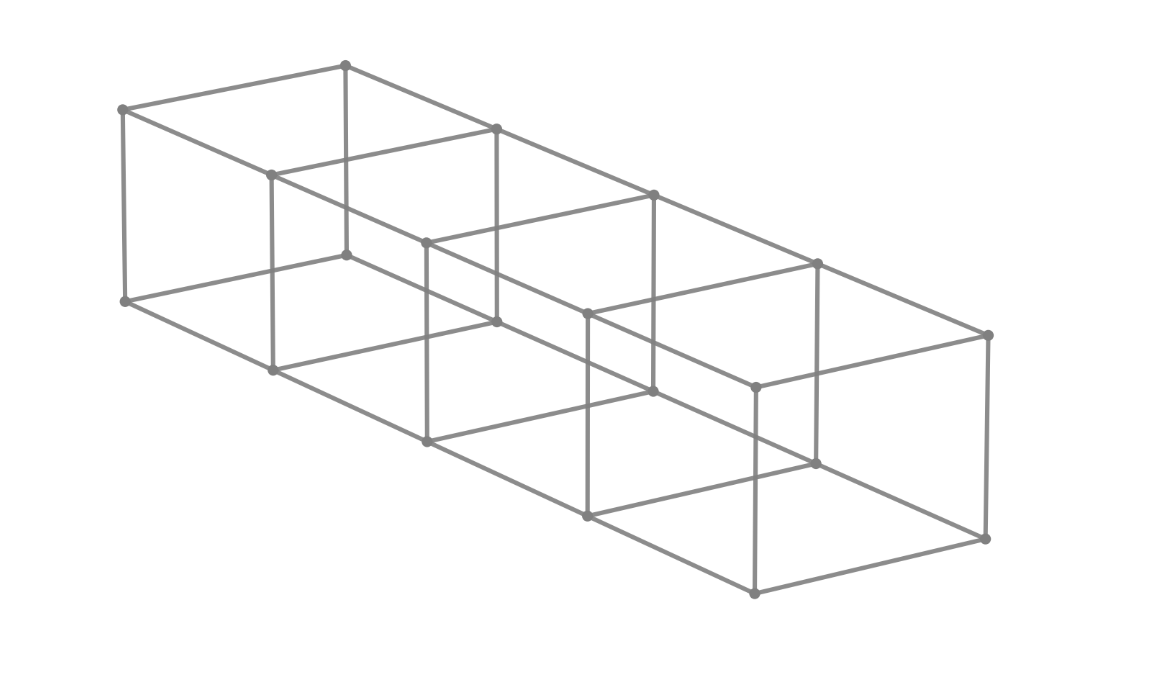}
    \caption{Cylindrical graph $G_{4,5}=C_4\times P_5$.}
    \label{fig:a}
  \end{subfigure}\hfill
  \begin{subfigure}{0.48\textwidth}
    \centering
    \includegraphics[width=\linewidth]{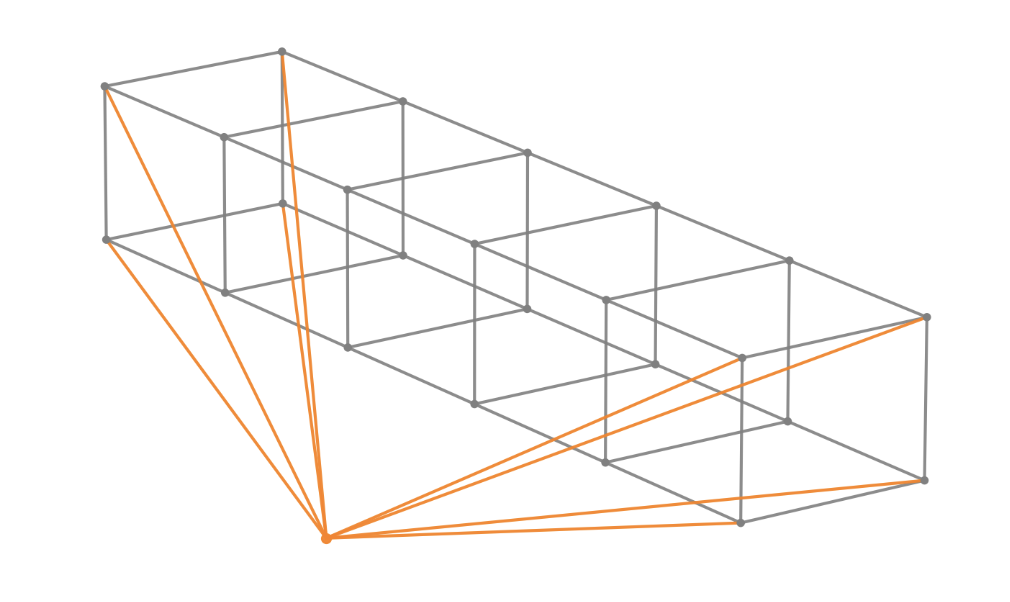}
    \caption{Cylindrical graph $G_{4,6}^{\mathrm s}$ with sink.}
    \label{fig:b}
  \end{subfigure}
  \caption{Cylindrical graphs.}
  \label{fig:ab}
\end{figure}

Let \(T\) be a spanning tree of \(G_{n,m}^{\mathrm{s}}\). For every
real vertex \(v\in V(G_{n,m})\), there is a unique path in \(T\) from
\(v\) to \(s\).  If the last real vertex of this path before \(s\) lies in \(R_0\), we
call \(v\) \emph{left-labelled}; if it lies in \(R_{m-1}\), we call
\(v\) \emph{right-labelled}.  We denote the corresponding vertex sets by
\(V_l(T)\) and \(V_r(T)\).  The left and right segments of \(T\) are the subgraphs of \(T\) spanned by the
paths from vertices of \(V_l(T)\) and \(V_r(T)\), respectively, to the sink
\(s\).

\begin{figure}[H]
        \centering
        \includegraphics[width=0.8\linewidth]{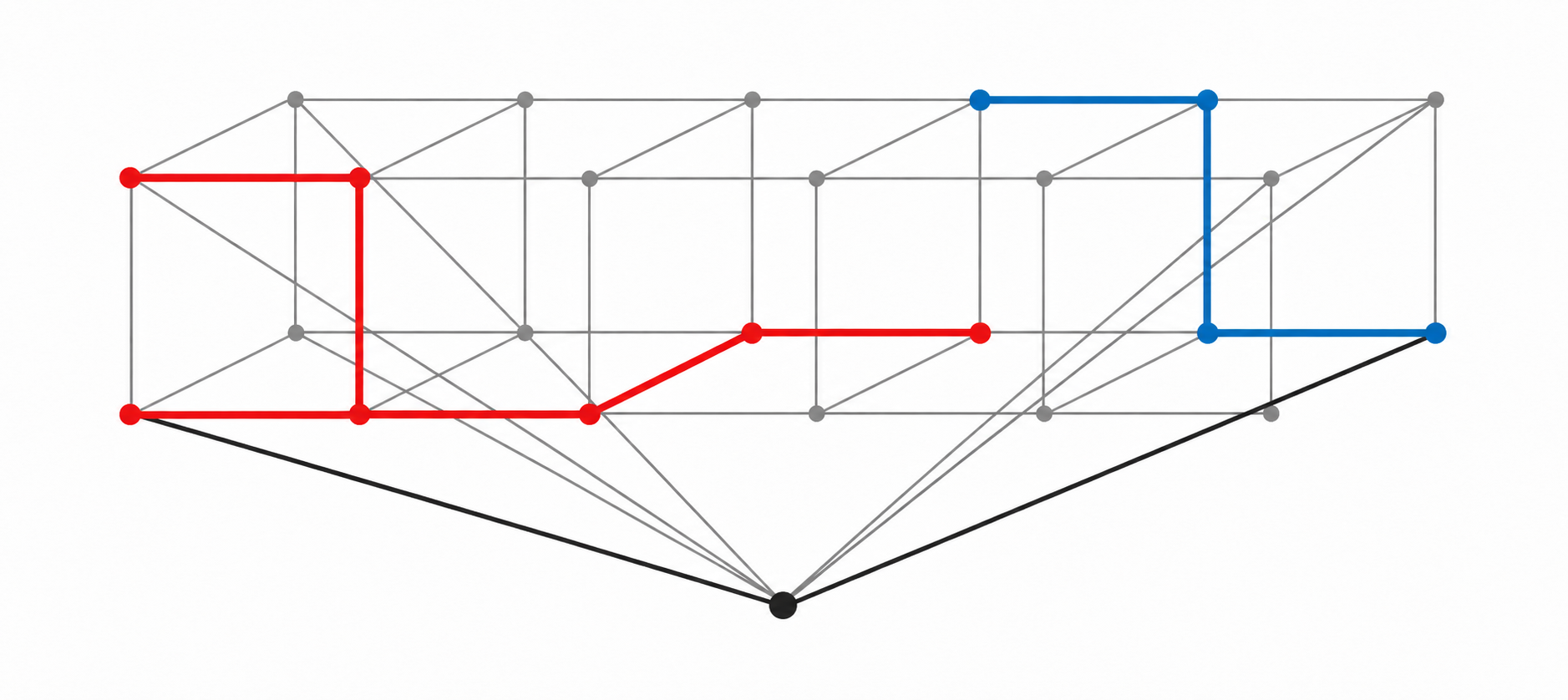}
        \caption{A tree subgraph of \(G_{4,6}^{\mathrm s}\) illustrating the
left and right segments. The sink and its incident edges are shown in
black. The left segment is highlighted in red, and the right segment
in blue.}
\label{fig:left_right}
    \end{figure}

\begin{definition}
Let \(G\) be either \(G_{n,m}\) or \(G_{n,m}^{\mathrm s}\), and let
\(T\) be a spanning tree of \(G\). A \emph{trunk} of \(T\) is a simple path \(L\subset T\)
such that
\[
L\cap R_k\neq\varnothing
\qquad\text{for every }k=0,\ldots,m-1.
\]
If \(G=G_{n,m}^{\mathrm s}\), then a \emph{sink-trunk} is a trunk
\(L\subset T\) such that \(s\in L\).
\end{definition}

\begin{definition}
Let \(L\subset T\) be a trunk. Every connected component \(B\) of
\(T\setminus V(L)\) is adjacent to a unique vertex
\(a(B)\in V(L)\). The corresponding \emph{branch off \(L\)} is the
subtree induced by
\[
V(B)\cup\{a(B)\}.
\]
Its length is
\[
\operatorname{len}(B)
:=
\max_{v\in V(B)}d_T(v,a(B)).
\]
Equivalently, the maximal branch length off \(L\) is
\[
\max_{v\in V(T)}d_T(v,L),
\qquad
d_T(v,L):=\min_{u\in V(L)}d_T(v,u).
\]
\end{definition}

\begin{definition}
A \emph{uniform spanning tree} (UST) of a finite connected graph \(G\)
is a spanning tree chosen uniformly from the set of all spanning trees
of \(G\).
\end{definition}

\begin{figure}[H]
    \centering
    \includegraphics[width=0.5\linewidth]{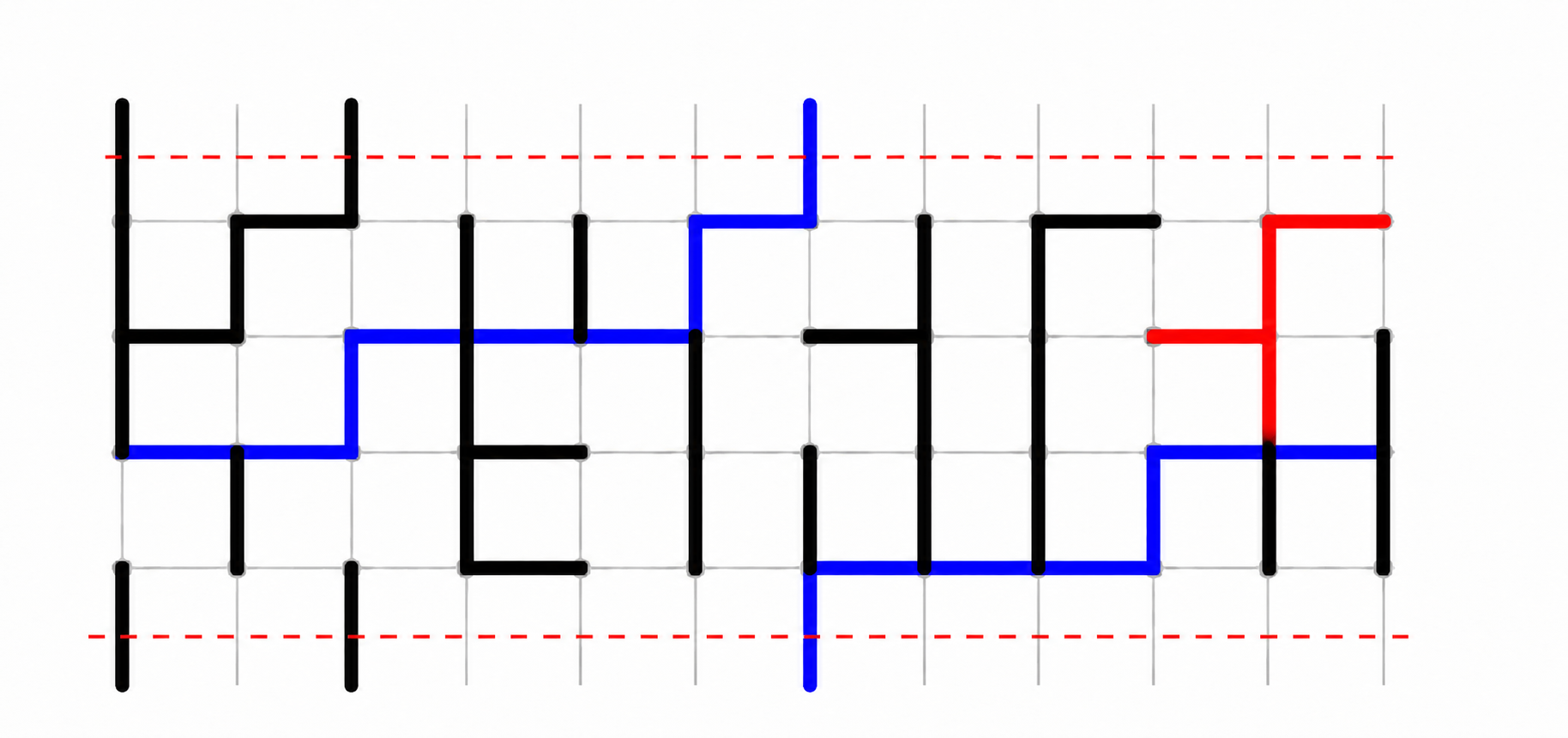}
    \caption{An unwrapped view of a spanning tree on the cylindrical graph \(G_{4,12}\).
The designated trunk is blue, the branch is red, and the remaining edges are black.
The top and bottom boundaries are identified; that is, we glue together the edges cut according to the hatching.}
    \label{fig:tree}
\end{figure}

% \begin{figure}[h]
%     \centering
%     \includegraphics[width=0.3\linewidth]{tree.png}
%     \caption{An unwrapped view of a spanning tree on the cylindrical graph \(G_{4,12}\).
% The designated trunk is blue, the branch is red, and the remaining edges are black.
% The top and bottom boundaries are identified; that is, we glue together the edges cut according to the hatching.}
%     \label{fig:tree}
% \end{figure}

\begin{theorem}[Exponential tail for branches off the Wilson trunk of a UST on a cylinder]\label{thm:main}
Let \(n\ge3\) and \(m\ge2\), and let \(G_{n,m}=C_n\times P_m\). Sample a uniform spanning tree $T$ on $G_{n,m}$ by Wilson's algorithm as follows: start with a root vertex on one boundary ring, and start the first random walk from a vertex on the other boundary ring. Let $L\subset T$ be the first loop-erased path produced by the algorithm. Then $L$ is a trunk, and for every $l\ge 0$,
\[
\mathbb{P}\!\left(\text{$T$ has a branch off $L$ of length at least } l\right)
\;\le\; C\,m\, (n-1)\,\theta^{\,l},
\]
where $C=C(n)>0$ and $\theta=\theta(n)\in(0,1)$ depend only on the circumference $n$ and are independent of $m$ and $l$.
\end{theorem}

\subsection{The wired cylinder}

We next state the corresponding result for the wired cylinder.

\begin{theorem}[Existence of a sink-trunk with short branches in the wired model]
\label{thm:exist-good-trunk-wired}
Fix \(n\ge3\). There exist constants \(C=C(n)<\infty\) and
\(\theta=\theta(n)\in(0,1)\), independent of \(m\), such that for
every \(m\ge2\) and every \(\ell\ge1\), the wired uniform spanning
tree \(T\) of \(G_{n,m}^{\mathrm s}\) satisfies
\[
\begin{aligned}
&\mathbb P\left(
\exists\text{ a sink-trunk }L\subset T
\text{ such that every branch off }L\text{ has length}<\ell
\right) \\
&\qquad\ge 1-Cm^3\theta^\ell .
\end{aligned}
\]
\end{theorem}

\begin{remark}
The estimate in Theorem~\ref{thm:exist-good-trunk-wired} should be interpreted
on the logarithmic scale in \(m\). Since the failure probability is bounded by \(Cm^3\theta^\ell\), the
theorem implies that for every \(A>0\), with
\[
\ell=
\left\lceil
\frac{3\log m+A}{|\log\theta|}
\right\rceil,
\]
one has
\[
\mathbb P\left(
\exists\text{ a sink-trunk }L\subset T
\text{ such that every branch off }L\text{ has length}<\ell
\right)
\ge
1-Ce^{-A}.
\]
Thus the result gives a sink-trunk whose maximal branch length is
\(O_n(\log m)\) with high probability. The constants in the exponential
increment estimate are uniform in \(m\), but the maximum over all possible
branches naturally introduces the logarithmic scale.
\end{remark}

\begin{definition}
Let \(T\) be a spanning tree of \(G_{n,m}^{\mathrm s}\), with
left- and right-labelled vertex sets \(V_l(T)\) and \(V_r(T)\) as
defined above. The \emph{LR--Slash} of \(T\) is
\[
\Slash_{\mathrm{LR}}(T)
:=
\bigl\{\{u,v\}\in E(G_{n,m}):
u\in V_l(T),\ v\in V_r(T)
\text{, or }
u\in V_r(T),\ v\in V_l(T)
\bigr\}.
\]
Its size is
\[
\bigl|\Slash_{\mathrm{LR}}(T)\bigr|
:=
\#\,\Slash_{\mathrm{LR}}(T).
\]
\end{definition}

\begin{figure}[H]
    \centering
    \includegraphics[width=0.5\linewidth]{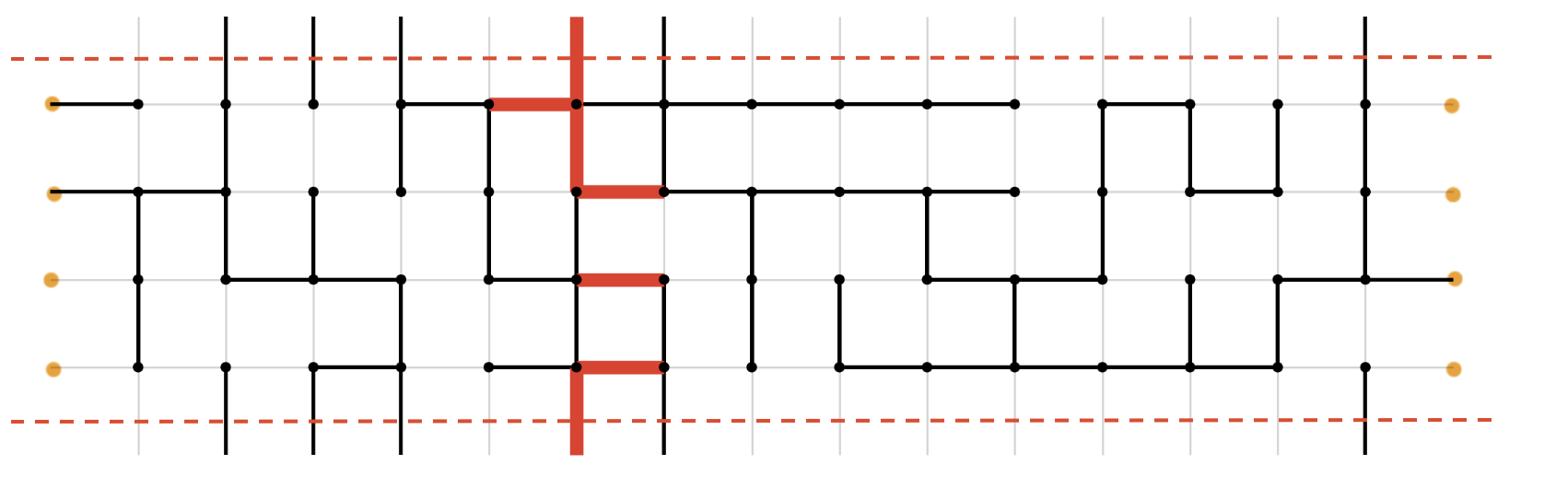}
    \caption{An unwrapped view of the wired cylindrical graph
\(G_{4,15}^{\mathrm s}\). The orange copies represent the single sink
\(s\). The red edges form the LR--Slash between the left and right
segments of the spanning tree.}
    \label{fig:slash}
\end{figure}

\begin{theorem}\label{thm:Slash}
Let \(n\ge3\) and \(m\ge2\), and let \(T\) be a uniform spanning tree on the
wired cylindrical graph \(G_{n,m}^{\mathrm s}\). There exist constants
\(C=C(n)>0\) and \(\delta=\delta(n)\in(0,1)\), depending only on the
circumference \(n\), such that for every integer \(l\ge1\),
\[
\mathbb{P}\!\left(\, \bigl|\Slash_{\mathrm{LR}}(T)\bigr|>l \,\right)\;\le\; C\,\delta^{\,l}.
\]
\end{theorem}

\begin{remark}[Dual interpretation of the LR--Slash]
\label{rem:dual-slash}
Consider the natural embedding of \(G_{n,m}\) in the cylinder, and let
\(G_{n,m}^*\) be its dual graph. For an edge
\(e\in E(G_{n,m})\), denote the corresponding dual edge by \(e^*\), and set
\[
\Slash_{\mathrm{LR}}^*(T)
:=
\{e^*:e\in\Slash_{\mathrm{LR}}(T)\}.
\]

The dual edge set \(\Slash_{\mathrm{LR}}^*(T)\) is the interface between the
left- and right-labelled vertices. Indeed, every path in \(G_{n,m}\) from
\(V_l(T)\) to \(V_r(T)\) contains an edge of
\(\Slash_{\mathrm{LR}}(T)\), and therefore crosses its dual edge set.

In general, \(\Slash_{\mathrm{LR}}^*(T)\) need not be a single simple cycle:
it may contain several components, and some components may meet the boundary
of the cylinder. Nevertheless, every simple noncontractible dual cycle
\[
\Gamma^*\subset\Slash_{\mathrm{LR}}^*(T)
\]
winds once around the cylinder. Equivalently, \(\Gamma^*\) represents a
generator of the first homology group of the cylinder.
\end{remark}

\section{Proofs}
\subsection{Wilson's algorithm}
Let $G=(V,E)$ be a finite connected  graph. 

\begin{definition}
A \emph{simple random walk} on \(G\) is a Markov chain \((X_k)_{k\ge0}\) with transition probabilities
\[
\mathbb{P}(X_{k+1}=v\mid X_k=u)
=\begin{cases}
1/\deg_G(u), & \text{if } uv\in E,\\[3pt]
0, & \text{otherwise.}
\end{cases}
\]
\end{definition}

\begin{definition}
Let \(\gamma=(x_0,x_1,\dots,x_l)\) be a finite path in \(G\).  
The \emph{loop-erased path} \(\mathrm{LE}(\gamma)\) is obtained by iteratively erasing cycles in chronological order: whenever the path first revisits a vertex, delete the entire loop formed between the two occurrences.  
The result is a simple path with the same start and end vertices as \(\gamma\).
\end{definition}

We generate uniform spanning trees using Wilson's
algorithm~\cite{Wilson1996}. Choose a root \(r\in V\) and an ordering
of \(V\setminus\{r\}\). Starting with the tree \(\{r\}\), repeatedly
choose the first vertex in the order that is not yet in the tree, run
simple random walk from it until the walk first hits the current tree,
and add the chronological loop-erasure of the walk. The resulting
spanning tree is uniform, and its law is independent of the chosen
ordering.

\begin{lemma}[Adaptive Wilson ordering, see also Chapter~4, \cite{lyons2016probability}]
\label{lem:adaptive-wilson}
Suppose that Wilson's algorithm on a finite connected graph \(G\) has
constructed a tree \(F\). Conditional on the current tree being \(F\), the
remaining starting vertices may be chosen successively by deterministic rules
depending on the tree revealed so far. Such adaptive choices do not change the
conditional law of the final spanning tree.
\end{lemma}

\begin{proof}
Fix a possible current tree \(F\). Conditional on \(F\), all random walks used
in the continuation of Wilson's algorithm are fresh, and the next starting
vertex selected by the adaptive rule is deterministic.

Wilson's algorithm started from the fixed initial tree \(F\) has a final law
independent of the deterministic order of the remaining vertices. Therefore
the first adaptive choice does not change the conditional law of the
completion. After the resulting loop-erased path has been added, the same
argument applies to the enlarged current tree. Iterating proves the claim.
\end{proof}

\subsection{Proof of Theorem~\ref{thm:main} }

Throughout the paper, the length \(|\beta|\) of a path \(\beta\)
means its number of edges.

\begin{lemma}[Wilson trunks and uniform increment tails]
\label{lem:uniform-wilson-increment}
Fix \(n\ge3\).

\begin{enumerate}
    \item Let Wilson's algorithm be run on \(G_{n,m}\), with the root
    on one boundary ring and the first random walk started from the
    other boundary ring. Then the first loop-erased path produced by
    the algorithm is a trunk.

    \item Let \(G\) be either \(G_{n,m}\) or
    \(G_{n,m}^{\mathrm s}\). In the wired case, assume that Wilson's
    algorithm is rooted at \(s\), so that \(s\in F\). Suppose that the
    current Wilson tree \(F\) intersects every ring
    \(R_0,\ldots,R_{m-1}\). Let \(\beta\) be the next nontrivial
    loop-erased increment. Then there exist constants
    \(C=C(n)<\infty\) and \(\theta=\theta(n)\in(0,1)\), independent
    of \(m\), \(F\), and the starting vertex, such that
    \[
    \mathbb P\bigl(|\beta|\ge\ell\mid F\bigr)
    \le
    C\theta^\ell
    \qquad(\ell\ge0).
    \]
\end{enumerate}
\end{lemma}

\begin{proof}
For the first assertion, the first loop-erased path connects the two
boundary rings. Since the longitudinal coordinate changes by at most
one along each edge, it intersects every ring
\(R_0,\ldots,R_{m-1}\), and is therefore a trunk.

For the second assertion, put
\[
M:=\left\lfloor\frac n2\right\rfloor.
\]
Since \(F\) intersects every ring, every real vertex is within graph
distance \(M\) of \(F\), by moving along its ring. All real vertices
of \(G_{n,m}\) and \(G_{n,m}^{\mathrm s}\) have degree at most \(4\).
Thus, from any real vertex, the random walk hits \(F\) within the next
\(M\) steps with probability at least
\[
p:=4^{-M},
\]
by following a fixed shortest path to \(F\).

Let \(\tau_F\) be the hitting time of \(F\). By the Markov property,
\[
\mathbb P(\tau_F>qM\mid F)
\le
(1-p)^q
\qquad(q\ge0).
\]
Since loop-erasure cannot increase path length,
\[
|\beta|\le\tau_F.
\]
The claimed estimate follows from this geometric tail after changing
constants.
\end{proof}

\begin{proof}[Proof of Theorem~\ref{thm:main}]
Condition on the first Wilson trunk \(L\), and fix
\(v\in V(G_{n,m})\setminus L\). By
Lemma~\ref{lem:adaptive-wilson}, after \(L\) has been created we may
start the next walk from \(v\), without changing the conditional law
given \(L\). The loop-erased path added at this step runs from \(v\)
to \(L\). In the final tree this path is the unique path from \(v\)
to \(L\), and therefore its length is \(d_T(v,L)\).

By Lemma~\ref{lem:uniform-wilson-increment}, uniformly in \(v\) and
\(L\),
\[
\mathbb P\bigl(d_T(v,L)\ge l\mid L\bigr)
\le
C(n)\theta(n)^l
\qquad(l\ge0).
\]

Assume first that \(l\ge1\). A branch off \(L\) has length at least
\(l\) precisely when some vertex outside \(L\) has tree-distance at
least \(l\) from \(L\). Since \(L\) intersects every ring, it contains
at least \(m\) vertices, and therefore
\[
|V(G_{n,m})\setminus L|\le nm-m=(n-1)m.
\]
Thus, conditionally on \(L\),
\[
\begin{aligned}
&\mathbb{P}\!\left(\text{\(T\) has a branch off \(L\) of length at least \(l\)}
\,\middle|\,L\right)\\
&\qquad\le
\sum_{v\in V(G_{n,m})\setminus L}
\mathbb P\bigl(d_T(v,L)\ge l\mid L\bigr)\\
&\qquad\le
(n-1)m\,C(n)\theta(n)^l .
\end{aligned}
\]
Taking expectation over \(L\) gives the desired estimate for \(l\ge1\).
The case \(l=0\) follows after increasing \(C(n)\), since the relevant
probability is at most \(1\).
\end{proof}

Initially we had a more involved proof, and we thank Yuval Peres for the simplified argument presented above.

\begin{lemma}[Existence of a sink-trunk]
\label{lem:sink-trunk-exists}
Every spanning tree of \(G_{n,m}^{\mathrm s}\) contains a sink-trunk.
\end{lemma}

\begin{proof}
Label each real vertex left or right according to the last edge of its
path to \(s\).  If all real vertices are left-labelled, choose \(x\in R_{m-1}\). The
path from \(x\) to \(s\) ends through \(R_0\); hence, since the
longitudinal coordinate changes by at most one along each edge, it
intersects every ring. This path is a sink-trunk. The all-right case is
symmetric.

Otherwise both labels occur. Let
\[
i_l:=\max\{k: V_l(T)\cap R_k\neq\varnothing\},
\qquad
i_r:=\min\{k: V_r(T)\cap R_k\neq\varnothing\}.
\]
Since every ring contains a real vertex, there is no empty ring between
the left-labelled and right-labelled regions; hence \(i_l+1\ge i_r\).
Choose \(x\in V_l(T)\cap R_{i_l}\) and
\(y\in V_r(T)\cap R_{i_r}\). The paths from \(x\) and \(y\) to \(s\)
meet only at \(s\): if they met earlier, their remaining path to \(s\)
would be common, forcing the same label. Thus their union is a simple path through \(s\). The path from \(x\)
to \(s\) intersects \(R_0,\ldots,R_{i_l}\), while the path from \(y\)
to \(s\) intersects \(R_{i_r},\ldots,R_{m-1}\). Since \(i_l+1\ge i_r\),
the union intersects every ring and is therefore a sink-trunk.
\end{proof}
We also use the stack representation of Wilson's algorithm and the
cycle-popping theorem; see, for example,
\cite[Chapter~4]{lyons2016probability}.  At every
real vertex \(v\), place an infinite stack of independent uniformly
chosen neighbours of \(v\). A walk follows the currently exposed stack
arrow at its present vertex, and whenever a directed cycle is formed,
the arrows belonging to that cycle are popped. The cycle-popping theorem
states that, for a fixed stack realization, the final oriented spanning
tree rooted at \(s\) is independent of the order in which legal cycles
are popped. Equivalently, Wilson's algorithm run in any deterministic
vertex order produces the same final tree from the same stacks.
\begin{lemma}[Fixed-stack path reconstruction]
\label{lem:stack-path-reconstruction}
Fix a stack realization of Wilson's cycle-popping algorithm, and let
\(T\) be the final oriented spanning tree rooted at \(s\). Run Wilson's
algorithm with the same stacks. Suppose that, at some stage, the current
oriented tree \(F\) is a rooted oriented subtree of \(T\). If a vertex
\(x\notin F\) is started at this stage, then the loop-erased path added by
Wilson's algorithm is the unique directed path in \(T\) from \(x\) to \(F\).
\end{lemma}

\begin{proof}
Complete the Wilson run after the path from \(x\) has been added.
Edges already added to the current tree are never changed during the
remainder of Wilson's algorithm. Hence the path added from \(x\) is
contained in the final oriented tree produced by this run.

By the cycle-popping theorem, the final tree obtained from the fixed
stacks is independent of the Wilson order, and is therefore \(T\).
Thus the path added from \(x\) is contained in \(T\). It starts at
\(x\), ends when it first reaches \(F\), and is directed towards
\(F\). Since \(T\) contains a unique directed path from \(x\) to the
rooted subtree \(F\), the added loop-erased path must be that path.
\end{proof}

\subsection{Proof of Theorem~\ref{thm:exist-good-trunk-wired}}
Throughout the wired case, Wilson's algorithm is always rooted at the sink
\(s\). Thus the resulting spanning tree of \(G_{n,m}^{\mathrm s}\) is viewed
as an oriented tree in which every real vertex has a unique directed path to
\(s\).

\begin{proof}[Proof of Theorem~\ref{thm:exist-good-trunk-wired}]
We use Wilson's algorithm rooted at \(s\), in its stack representation.
For a fixed stack realization, the cycle-popping theorem says that the final
oriented spanning tree rooted at \(s\) is independent of the order in which
cycles are popped. Equivalently, Wilson's algorithm run in any deterministic
vertex order produces the same final tree from the same stacks.

Fix an ordered triple
\[
(x,y,z)\in (V(G_{n,m})\cup\{s\})^2\times V(G_{n,m}).
\]
Run Wilson's algorithm in the order
\[
x,\quad y,\quad z,\quad \text{then all remaining real vertices},
\]
skipping \(s\) and vertices already in the current tree. Let \(F_{x,y}\) be
the tree after the walks started from \(x\) and \(y\) have been processed. If
\(F_{x,y}\) is a sink-trunk, denote it by \(L_{x,y}\).

Let \(A_{x,y,z}(\ell)\) be the event that \(F_{x,y}\) is a sink-trunk and
that the Wilson increment added when \(z\) is processed has loop-erased
length at least \(\ell\). If \(z\in F_{x,y}\), this increment has length zero.

Let \(C_0=C_0(n)<\infty\) and
\(\theta_0=\theta_0(n)\in(0,1)\) be the constants from
Lemma~\ref{lem:uniform-wilson-increment}. Let
\(\mathcal H_{x,y}\) be the sigma-field generated by the stack entries
revealed while processing the walks started from \(x\) and \(y\), together
with the resulting tree \(F_{x,y}\).

Conditional on \(\mathcal H_{x,y}\), the unused tails of all stacks remain
independent and have their original distributions. This is the usual
deferred-decisions property for independent i.i.d.\ stacks: conditional on
the adaptively revealed finite prefixes, the unrevealed suffixes remain
mutually independent and retain their original product distribution.
Consequently, if \(z\notin F_{x,y}\), the continuation started from \(z\),
up to its first hit on \(F_{x,y}\), has the law of a simple random walk
started from \(z\) and stopped on hitting \(F_{x,y}\).

On the event that \(F_{x,y}\) is a sink-trunk,
Lemma~\ref{lem:uniform-wilson-increment} therefore gives
\[
\mathbb P\bigl(
A_{x,y,z}(\ell)
\,\bigm|\,
\mathcal H_{x,y}
\bigr)
\le
C_0\theta_0^\ell.
\]
If \(F_{x,y}\) is not a sink-trunk, then
\(A_{x,y,z}(\ell)\) does not occur. Hence
\[
\mathbb P\bigl(
A_{x,y,z}(\ell)
\,\bigm|\,
\mathcal H_{x,y}
\bigr)
\le
\mathbf 1_{\{F_{x,y}\text{ is a sink-trunk}\}}
C_0\theta_0^\ell
\le
C_0\theta_0^\ell.
\]
Taking expectations yields
\[
\mathbb P\bigl(A_{x,y,z}(\ell)\bigr)
\le C_0\theta_0^\ell
\]
uniformly in \(m,x,y,z\).

We now prove the deterministic implication behind the union bound. Fix the
stacks, and let \(T\) be the final tree. Suppose that \(T\) has no sink-trunk
whose branches all have length \(<\ell\). By
Lemma~\ref{lem:sink-trunk-exists}, choose a sink-trunk \(L\subset T\). Since \(L\) does not have all its branches of length \(<\ell\), there
is a real vertex \(z\) such that
\[
d_T(z,L)\ge \ell .
\]

Let \(x,y\) be the two endpoints of \(L\), allowing one endpoint to be \(s\).
Run Wilson's algorithm from the same stacks in the order
\[
x,\quad y,\quad z,\quad \text{then all remaining real vertices}.
\]
If neither endpoint is \(s\), then \(s\) lies in the interior of \(L\), and
\(L\) is the union of the two directed arms from \(x\) and \(y\) to \(s\).
These arms meet only at \(s\), since \(L\) is a simple path.

By Lemma~\ref{lem:stack-path-reconstruction}, the walks started from the real
endpoint(s) of \(L\) reconstruct the directed arm(s) of \(L\) in \(T\). Thus the tree
obtained after processing \(x\) and \(y\), with \(s\) skipped if it is one of
the endpoints, is exactly \(L\), so \(L_{x,y}=L\).

Again by Lemma~\ref{lem:stack-path-reconstruction}, the path then added from
\(z\) is the unique path in \(T\) from \(z\) to \(L\). Its length is
\(d_T(z,L)\ge\ell\), so \(A_{x,y,z}(\ell)\) occurs.

Thus, for each stack realization,
\[
\left\{
\text{no sink-trunk in }T\text{ has all branches of length}<\ell
\right\}
\subset
\bigcup_{x,y,z} A_{x,y,z}(\ell),
\]
where \(x,y\in V(G_{n,m})\cup\{s\}\) and \(z\in V(G_{n,m})\).

There are at most \((nm+1)^2nm\le C_1(n)m^3\) such triples. Hence, by the
union bound and the estimate above,
\[
\mathbb P\left(
\text{no sink-trunk has all branches of length}<\ell
\right)
\le
C_1(n)m^3 C_0(n)\theta_0(n)^\ell .
\]
Absorbing \(C_1(n)C_0(n)\) into a new constant \(C(n)\) and setting
\(\theta(n)=\theta_0(n)\), we obtain
\[
\mathbb P\left(
\text{no sink-trunk has all branches of length}<\ell
\right)
\le
C(n)m^3\theta(n)^\ell .
\]
This is the desired existence estimate.
\end{proof}

\subsection{Proof of the LR--Slash estimate}

{\bf Main idea of the proof.} The proof below identifies a narrow random transition zone between the
left- and right-labelled parts of the Wilson tree. Using an adaptive
Wilson ordering, we start successive walks from middle rings between
the current left and right fronts; the gap between the fronts then
decreases geometrically, while the probability that the stopping
increment overshoots the opposite front by \(q\) rings decays
exponentially in \(q\). After the fronts meet, the revealed tree
intersects every ring, and the transition zone has exponentially
controlled width. A vertex lying \(d\) rings beyond an
\(h\)-neighbourhood of this zone can receive the wrong label only if
its random walk crosses \(h+d\) successive rings, each containing a
vertex of the correct-labelled part, without hitting that part; this
probability is exponentially small in \(h+d\). The transition zone has exponentially controlled width, and summing
the wrong-label estimate over the \(O(n)\) vertices on each ring shows
that the probability of an LR--Slash edge lying more than \(h\) rings
from this zone is \(O(\rho^h)\). Hence, except on an event of
exponentially small probability in \(h\), the entire LR--Slash is
contained in an \(O_n(h)\)-edge neighbourhood. Combining this
localization with the exponential overshoot bound and taking \(h\)
proportional to the desired slash-size threshold yields the
exponential tail for
\(
\bigl|\Slash_{\mathrm{LR}}(T)\bigr|.
\)

\begin{proof}[Proof of Theorem~\ref{thm:Slash}]
Throughout the proof, \(n\) is fixed and all constants may depend on
\(n\), but not on \(m\). Since \(G_{n,2}^{\mathrm s}\) has only
\(O_n(1)\) edges, the case \(m=2\) is absorbed into the
multiplicative constant. Hence we assume \(m\ge3\).

We run Wilson's algorithm on \(G_{n,m}^{\mathrm s}\), rooted at its
single sink \(s\), using the adaptive ordering permitted by
Lemma~\ref{lem:adaptive-wilson}. Let
\[
\pi(i,j):=j
\]
be the longitudinal coordinate.

For the current Wilson tree \(F^t\), let \(F_l^t\) and \(F_r^t\) be
its left- and right-labelled real vertices ($s$ is not labelled). An increment inherits the
label of the real vertex at which it attaches; if it hits \(s\), its
label is determined by whether its last real vertex lies in \(R_0\)
or \(R_{m-1}\).

For a random walk \(X\), let \(\tau_l^t\) be the first time that \(X\)
hits \(F_l^t\) or enters \(s\) from \(R_0\), and let \(\tau_r^t\) be
defined symmetrically. The increment attaches to the left or right
part according as
\[
\tau_l^t<\tau_r^t
\qquad\text{or}\qquad
\tau_r^t<\tau_l^t.
\]

Define
\[
i_l^t
:=
\max\{k:F_l^t\cap R_k\neq\varnothing\},
\qquad
i_r^t
:=
\min\{k:F_r^t\cap R_k\neq\varnothing\},
\]
with the conventions
\[
\max\varnothing=-1,
\qquad
\min\varnothing=m.
\]
Thus initially
\[
i_l^0=-1,
\qquad
i_r^0=m.
\]

We use the front property
\[
F_l^t\cap R_k\neq\varnothing
\qquad(0\le k\le i_l^t),
\]
whenever \(i_l^t\ge0\), and
\[
F_r^t\cap R_k\neq\varnothing
\qquad(i_r^t\le k\le m-1),
\]
whenever \(i_r^t\le m-1\). This follows by induction: any increment
that extends one of the fronts joins its new extreme ring to the old
part of the same label, or to the corresponding boundary entrance.
Since the longitudinal coordinate changes by at most one along an
edge, the increment visits every intermediate ring.

We shall repeatedly use the following elementary ring-barrier
estimate. There exists \(\rho=\rho(n)\in(0,1)\) such that, if each of
\(q\) consecutive rings contains a vertex of a set \(H\), then the
probability that a random walk crosses all \(q\) rings in either
longitudinal direction without hitting \(H\) is at most
\[
\rho^q.
\]
Indeed, on first entering any one of these rings, the walk has
probability at least
\[
p_0:=4^{-\lfloor n/2\rfloor}>0
\]
to follow a shortest path along the ring to \(H\) before making
another longitudinal move. Successive applications of the strong
Markov property give the claim with \(\rho=1-p_0\).

An edge belongs to \(\Slash_{\mathrm{LR}}(T)\) exactly when its
endpoints have opposite labels.

We first use an adaptive exploration to locate the transition zone.
While
\[
i_l^t+1<i_r^t,
\]
put
\[
D_t:=i_r^t-i_l^t,
\qquad
c_t:=
\left\lfloor
\frac{i_l^t+i_r^t}{2}
\right\rfloor,
\]
and start the next walk from any vertex \(x_t\in R_{c_t}\).

Since \(c_t\) lies strictly between the fronts, \(R_{c_t}\) contains
no vertex of the current tree. Let \(\beta_t\) be the resulting
loop-erased increment. Stop at the first time \(\tau\) for which
\[
i_l^{\tau+1}+1\ge i_r^{\tau+1}.
\]

This stopping time is finite. Indeed, at every non-stopping step, an
increment attaching to the left part raises the left front to at least
\(c_t\), while an increment attaching to the right part lowers the
right front to at most \(c_t\). Hence
\[
D_{t+1}
\le
\left\lceil\frac{D_t}{2}\right\rceil.
\]
A non-stopping step is possible only when \(D_t\ge3\), and then
\[
\left\lceil\frac{D_t}{2}\right\rceil<D_t.
\]
Thus the integer gap strictly decreases until the exploration stops,
so \(\tau<\infty\).

For a path \(\beta\), define
\[
r(\beta):=
\max\{\pi(v):v\in V(\beta)\cap V(G_{n,m})\},
\]
and
\[
l(\beta):=
\min\{\pi(v):v\in V(\beta)\cap V(G_{n,m})\}.
\]
The overshoot of the stopping increment is
\[
W_\tau
:=
\begin{cases}
\bigl(r(\beta_\tau)-i_r^\tau\bigr)_+,
&\text{if \(\beta_\tau\) attaches to the left part},\\[2mm]
\bigl(i_l^\tau-l(\beta_\tau)\bigr)_+,
&\text{if \(\beta_\tau\) attaches to the right part}.
\end{cases}
\]

\begin{lemma}[One-step overshoot estimate]
Assume that the exploration has not stopped before step \(t\). For an
integer \(q\ge1\), define
\[
E_t^+(q)
:=
\left\{
\beta_t\text{ attaches to the left part and }
r(\beta_t)\ge i_r^t+q
\right\},
\]
and
\[
E_t^-(q)
:=
\left\{
\beta_t\text{ attaches to the right part and }
l(\beta_t)\le i_l^t-q
\right\}.
\]
There is a constant \(C_1=C_1(n)>0\) such that
\[
\mathbb P\bigl(E_t^\pm(q)\mid F^t\bigr)
\le
\frac{C_1}{D_t}\rho^q,
\]
where \(\rho=\rho(n)\in(0,1)\) is the constant in the ring-barrier
estimate.
\end{lemma}

\begin{proof}
We prove the estimate for \(E_t^+(q)\); the other case is symmetric.
If \(i_r^t+q>m-1\), then \(E_t^+(q)\) is empty. Assume therefore that
\[
i_r^t+q\le m-1.
\]

Let \(X\) be the random walk from \(x_t\). On \(E_t^+(q)\), the walk
reaches \(R_{i_r^t+q}\) before hitting the right part and subsequently
hits the left part before the right part. Hence
\[
\begin{aligned}
\mathbb P\bigl(E_t^+(q)\mid F^t\bigr)
&\le
\mathbb P_{x_t}\bigl(
\tau_{R_{i_r^t+q}}<\tau_r^t
\bigr)\\
&\quad\times
\sup_{z\in R_{i_r^t+q}}
\mathbb P_z\bigl(
\tau_l^t<\tau_r^t
\bigr).
\end{aligned}
\]

By the front property, each of the \(q\) rings
\[
R_{i_r^t},R_{i_r^t+1},\ldots,R_{i_r^t+q-1}
\]
contains a right-labelled vertex. The ring-barrier estimate therefore
gives
\[
\mathbb P_{x_t}\bigl(
\tau_{R_{i_r^t+q}}<\tau_r^t
\bigr)
\le
\rho^q.
\]

It remains to estimate the second factor. A walk started to the right
of \(R_{i_r^t}\) must reach \(R_{i_r^t}\) before it can hit the left
part. Choose a right-labelled vertex in \(R_{i_r^t}\). At every visit
to this ring, the walk has probability at least \(p=p(n)>0\) to hit
the right-labelled part before its next longitudinal departure.
Hence the expected number \(N\) of longitudinal departures from
\(R_{i_r^t}\) before absorption is at most \(p^{-1}\).

A departure to the right cannot reach the left part before returning
to \(R_{i_r^t}\) or hitting the right part. After a departure to the
left, the longitudinal coordinate, observed only when it changes, is
a simple symmetric walk started from \(i_r^t-1\). If \(i_l^t\ge0\),
gambler's ruin gives
\[
\mathbb P_{i_r^t-1}\bigl(
\tau_{i_l^t}<\tau_{i_r^t}
\bigr)
=
\frac1{D_t}.
\]
If \(i_l^t=-1\), reaching the left sink entrance requires first
reaching \(R_0\), and
\[
\mathbb P_{i_r^t-1}\bigl(
\tau_0<\tau_{i_r^t}
\bigr)
=
\frac1{i_r^t}
\le
\frac2{D_t}.
\]
Consequently,
\[
\sup_{w\in R_{i_r^t}}
\mathbb P_w\bigl(
\tau_l^t<\tau_r^t
\bigr)
\le
\frac{2\mathbb E N}{D_t}
\le
\frac{C(n)}{D_t}.
\]

Combining the two estimates proves
\[
\mathbb P\bigl(E_t^+(q)\mid F^t\bigr)
\le
\frac{C_1(n)}{D_t}\rho^q.
\]
\end{proof}

\begin{lemma}[Overshoot at the end of the first phase]
There exist constants \(C_2=C_2(n)>0\) and \(\rho_2=\rho_2(n)\in(0,1)\),
independent of \(m\), such that for every integer \(q\ge1\),
\[
\mathbb P(W_\tau\ge q)\le C_2\rho_2^q.
\]
\end{lemma}

\begin{proof}
If \(W_\tau\ge q\), then either \(E_\tau^+(q)\) or
\(E_\tau^-(q)\) occurs. Hence
\[
\begin{aligned}
\mathbb P(W_\tau\ge q)
&\le
\sum_{t\ge0}
\mathbb E\left[
\mathbf 1_{\{t\le\tau\}}
\mathbb P\bigl(
E_t^+(q)\cup E_t^-(q)\mid F^t
\bigr)
\right]\\
&\le
C(n)\rho^q
\mathbb E\left[
\sum_{t=0}^{\tau}\frac1{D_t}
\right].
\end{aligned}
\]

At every non-stopping step,
\[
D_{t+1}
\le
\left\lceil\frac{D_t}{2}\right\rceil.
\]
Since \(D_\tau\ge2\), reading the sequence backwards gives
\[
D_{\tau-j}\ge2^j+1
\qquad(0\le j\le\tau).
\]
Therefore
\[
\sum_{t=0}^{\tau}\frac1{D_t}
\le
\sum_{j=0}^{\infty}\frac1{2^j+1}
<\infty.
\]
It follows that
\[
\mathbb P(W_\tau\ge q)
\le
C_2(n)\rho^q.
\]
Thus the lemma holds with \(\rho_2:=\rho\).
\end{proof}

At the end of the adaptive exploration, put
\[
F^*:=F^{\tau+1},
\qquad
i_l^*:=i_l^{\tau+1},
\qquad
i_r^*:=i_r^{\tau+1}.
\]
By the stopping rule,
\[
i_l^*+1\ge i_r^*.
\]
Together with the front property, this implies that \(F^*\) intersects
every real ring.

Let
\[
a:=\min(i_l^*,i_r^*),
\qquad
b:=\max(i_l^*,i_r^*).
\]
Then
\[
b-a\le W_\tau+1.
\]
Indeed, suppose that the stopping increment attaches to the left part.
Then
\[
i_r^*=i_r^\tau,
\qquad
i_l^*=r(\beta_\tau).
\]
If \(i_l^*\le i_r^*\), the stopping rule gives
\(i_r^*-i_l^*\le1\); if \(i_l^*>i_r^*\), then
\[
i_l^*-i_r^*
=
r(\beta_\tau)-i_r^\tau
=
W_\tau.
\]
The case of attachment to the right part is symmetric.

By the definitions of the two fronts, \(F^*\) contains no
left-labelled vertex to the right of \(b\), and no right-labelled
vertex to the left of \(a\).

For an integer \(h\ge1\), define
\[
V_h
:=
\bigcup_{\substack{0\le k\le m-1\\
\operatorname{dist}(k,[a,b])\le h}}
R_k,
\]
where \(\operatorname{dist}\) denotes the usual distance in
\(\mathbb R\). Let
\[
\mathcal E_h
:=
\{uv\in E(G_{n,m}):u\in V_h\text{ or }v\in V_h\}.
\]
The set \(V_h\) contains at most
\[
b-a+2h+1
\le
W_\tau+2h+2
\]
rings. Hence
\[
|\mathcal E_h|
\le
C_E(n)(W_\tau+2h+2).
\]
In particular, on \(\{W_\tau\le h\}\),
\[
|\mathcal E_h|
\le
C_E'(n)h.
\]

\begin{lemma}[Localization of the final slash]
\label{lem:slash-localization}
There exist constants \(C_4=C_4(n)>0\) and
\(\theta_4=\theta_4(n)\in(0,1)\), independent of \(m\), such that
for every integer \(h\ge1\),
\[
\mathbb P\left(
\Slash_{\mathrm{LR}}(T)\not\subset\mathcal E_h
\,\middle|\,
F^*
\right)
\le
C_4\theta_4^h.
\]
\end{lemma}

\begin{proof}
Fix \(F^*\). Consider first a real vertex \(x\) lying to the right of
\(V_h\). Then, for some integer \(d\ge1\),
\[
\pi(x)=b+h+d.
\]
If \(x\in F^*\), it is already right-labelled. Otherwise, conditional
on \(F^*\), Lemma~\ref{lem:adaptive-wilson} allows us to process \(x\)
next without changing the conditional law of the completed tree.

For \(x\) to become left-labelled, its random walk must cross the
\(h+d\) rings
\[
R_{b+1},R_{b+2},\ldots,R_{b+h+d}
\]
without hitting the right-labelled part of \(F^*\). Each of these
rings contains a right-labelled vertex by the front property.
Therefore the ring-barrier estimate gives
\[
\mathbb P\bigl(
x\text{ is left-labelled}\mid F^*
\bigr)
\le
\rho^{h+d}.
\]
Symmetrically, if
\[
\pi(x)=a-h-d,
\qquad d\ge1,
\]
then
\[
\mathbb P\bigl(
x\text{ is right-labelled}\mid F^*
\bigr)
\le
\rho^{h+d}.
\]

If
\[
\Slash_{\mathrm{LR}}(T)\not\subset\mathcal E_h,
\]
then some LR--Slash edge has both endpoints outside \(V_h\). Since
adjacent vertices lie on the same or neighboring rings, both endpoints
lie on the same side of \(V_h\). One endpoint must therefore have the
wrong label: either a left-labelled vertex lies to the right of
\(V_h\), or a right-labelled vertex lies to its left.

For every \(d\ge1\), there are at most \(n\) vertices at distance
\(d\) on either side. Hence
\[
\begin{aligned}
\mathbb P\left(
\Slash_{\mathrm{LR}}(T)\not\subset\mathcal E_h
\,\middle|\,
F^*
\right)
&\le
2n\sum_{d\ge1}\rho^{h+d}\\
&=
\frac{2n\rho}{1-\rho}\,\rho^h.
\end{aligned}
\]
This proves the lemma.
\end{proof}

We now finish the proof. Let \(\ell\ge1\), and choose
\[
h:=\left\lfloor
\frac{\ell}{2C_E'(n)}
\right\rfloor.
\]
For the bounded values of \(\ell\) for which \(h=0\), the conclusion
follows after increasing \(C(n)\). We may therefore assume \(h\ge1\).

If \(W_\tau\le h\) and
\[
\Slash_{\mathrm{LR}}(T)\subset\mathcal E_h,
\]
then
\[
|\Slash_{\mathrm{LR}}(T)|
\le
|\mathcal E_h|
\le
C_E'(n)h
\le
\ell.
\]
Therefore
\[
\{|\Slash_{\mathrm{LR}}(T)|>\ell\}
\subseteq
\{W_\tau>h\}
\cup
\{\Slash_{\mathrm{LR}}(T)\not\subset\mathcal E_h\}.
\]
Using the overshoot estimate and
Lemma~\ref{lem:slash-localization},
\[
\begin{aligned}
\mathbb P\bigl(
|\Slash_{\mathrm{LR}}(T)|>\ell
\bigr)
&\le
\mathbb P(W_\tau>h)\\
&\quad+
\mathbb E\left[
\mathbb P\left(
\Slash_{\mathrm{LR}}(T)\not\subset\mathcal E_h
\,\middle|\,
F^*
\right)
\right]\\
&\le
C_2\rho_2^h+C_4\theta_4^h.
\end{aligned}
\]
Since \(h\) is proportional to \(\ell\), there exist constants
\(C=C(n)>0\) and \(\delta=\delta(n)\in(0,1)\), independent of \(m\),
such that
\[
\mathbb P\bigl(
|\Slash_{\mathrm{LR}}(T)|>\ell
\bigr)
\le
C\delta^\ell.
\]
Together with the already absorbed case \(m=2\), this proves
Theorem~\ref{thm:Slash}.
\end{proof}

\section{Disclosures}
\subsection{Conflict of interest statement}
No conflict of interest.

\subsection{Data access statement}
No datasets were generated or analyzed as part of the results of this article.

\subsection{Ethics statement}
Not applicable.

\subsection{Funding statement}
This work received no external funding.

\subsection{Acknowledgment}
We thank the referee for careful reading of the manuscript and suggesting many improvements.

%\bibliographystyle{plain} 
%\bibliography{name} 

\begin{thebibliography}{10}

\bibitem{AthreyaJarai2004}
Siva~R. Athreya and Antal~A. J{\'a}rai.
\newblock Infinite volume limit for the stationary distribution of abelian
  sandpile models.
\newblock {\em Communications in Mathematical Physics}, 249(1):197--213, 2004.

\bibitem{Dhar1999AbelianSandpile}
Deepak Dhar.
\newblock The abelian sandpile and related models.
\newblock {\em Physica A}, 263(1-4):4--25, 1999.

\bibitem{EckmannNagnibedaPerriard2023}
Jean-Pierre Eckmann, Tatiana Nagnibeda, and Aymeric Perriard.
\newblock Abelian sandpiles on cylinders.
\newblock {\em J. Phys. A: Mathematical and Theoretical}, 56(17):175001, 2023.

\bibitem{IzmailianKenna2015}
Nickolay~Sh. Izmailian and Ralph Kenna.
\newblock Exact finite-size corrections for the spanning-tree model under
  different boundary conditions.
\newblock {\em Physical Review E}, 91(2):022129, 2015.

\bibitem{KenyonWilson2015}
Richard~W. Kenyon and David~B. Wilson.
\newblock Spanning trees of graphs on surfaces and the intensity of loop-erased
  random walk on planar graphs.
\newblock {\em Journal of the American Mathematical Society}, 28(4):985--1030,
  2015.

\bibitem{lyons2016probability}
Russell Lyons and Yuval Peres.
\newblock {\em Probability on Trees and Networks}, volume~42 of {\em Cambridge
  Series in Statistical and Probabilistic Mathematics}.
\newblock Cambridge University Press, New York, 2016.
\newblock Available at \url{https://rdlyons.pages.iu.edu/}.

\bibitem{SchmidtVerbitskiy2009}
Klaus Schmidt and Evgeny Verbitskiy.
\newblock Abelian sandpiles and the harmonic model.
\newblock {\em Communications in Mathematical Physics}, 292(3):721--759, 2009.

\bibitem{Wilson1996}
David~B. Wilson.
\newblock Generating random spanning trees more quickly than the cover time.
\newblock In {\em Proceedings of the Twenty-Eighth Annual ACM Symposium on
  Theory of Computing (STOC '96)}, pages 296--303, New York, NY, USA, 1996.
  ACM.

\bibitem{Windisch2010}
David Windisch.
\newblock Random walks on discrete cylinders with large bases and random
  interlacements.
\newblock {\em The Annals of Probability}, 38(2):841--895, 2010.

\bibitem{YanZhang2011}
Weigen Yan and Fuji Zhang.
\newblock Enumeration of spanning trees of graphs with rotational symmetry.
\newblock {\em Journal of Combinatorial Theory, Series A}, 118(4):1270--1290,
  2011.

\bibitem{ZhangLuJin2024}
Jingyuan Zhang, Fuliang Lu, and Xian'an Jin.
\newblock Counting spanning trees of $(1,n)$-periodic graphs.
\newblock {\em Discrete Applied Mathematics}, 340:63--76, 2024.

\end{thebibliography}

\end{document}